\documentclass[12pt]{amsart}
\usepackage{amsfonts}

\newtheorem{theorem}{Theorem}[section]
\newtheorem{corollary}[theorem]{Corollary}

\newtheorem{lemma}[theorem]{Lemma}
\newtheorem{proposition}[theorem]{Proposition}
\newtheorem{example}[theorem]{Example}
\newtheorem{remark}[theorem]{Remark}

\def\NN{\hbox{\sf I\kern-.13em\hbox{N}}}
\def\RR{\hbox{\sf I\kern-.14em\hbox{R}}}

\def\Cc{\hbox{\sf C\kern -.47em {\raise .48ex \hbox{$\scriptscriptstyle |$}}
   \kern-.5em {\raise .48ex \hbox{$\scriptscriptstyle |$}} }}

\begin{document}

\baselineskip 8 mm

\title{On the positive commutator in the radical}
\author{Marko Kandi\'c, Klemen 	\v Sivic}
% \date{\thanks{ } June 27, 2005}
\date{\today}

\begin{abstract}
\baselineskip 7mm
In this paper we prove that a positive commutator between a positive compact operator $A$ and a positive operator $B$ is in the radical of the Banach algebra generated by 
$A$ and $B$.  Furthermore, on every at least three-dimensional Banach lattice we construct finite rank operators $A$ and $B$ satisfying $AB\geq BA\geq 0$ such that the commutator $AB-BA$ is not contained in the radical of the Banach algebra generated by $A$ and $B$. These two results now completely answer to two open questions published in \cite{BDFRZ}. We also obtain relevant results in the case of the Volterra and the Donoghue operator. 
\end{abstract}

\maketitle

\noindent
{\it Math. Subj. Classification (2010)}: 47A10; 47A46; 47B07; 47B47. 
 
\baselineskip 7.5mm

\section{Introduction}
\vspace{5mm}

Positive commutators of positive operators on Banach lattices were first considered in \cite{BDFRZ}. The authors proved  \cite[Theorem 2.2]{BDFRZ} that a positive commutator $AB-BA$ of positive compact operators $A$ and $B$ is necessarily quasi-nilpotent and furthermore, it is contained in the radical of the Banach algebra generated by $A$ and $B$. They posed a question if the same is still true if we assume that only one of the operators $A$ and $B$ is compact. 
The first part of the problem was independently solved by Drnov\v sek \cite{RD} and Gao \cite{Gao}. Drnov\v sek proved even more. 

\begin {theorem}\cite[Theorem 3.1]{RD}\label{roman I}
Let $A$ and $B$ be bounded operators on a Banach lattice $E$ such that $AB\geq BA\geq 0$ and $AB$ is power-compact. Then $AB-BA$ is an ideal-triangularizable quasi-nilpotent operator. 
\end {theorem}

This result is an improvement of \cite[Theorem 2.4]{BDFRZ} where the authors considered operators from appropriate Schatten ideals. In this paper 
we prove that a positive commutator $AB-BA$ of positive operators $A$ and $B$ is in the radical of the Banach algebra generated by $A$ and $B$ whenever one of the operators $A$ and $B$ is compact. We actually prove that $AB-BA$ is contained in the radical of the possibly bigger algebra that contains the Banach algebra generated by $A$ and $B$. 
This result now completely solves the open problem \cite[Open questions 3.7 (1)]{BDFRZ}. We also prove that on every at least three-dimensional Banach lattice $E$ there exist finite rank operators $A$ and $B$ on $E$ satisfying
$AB\geq BA\geq 0$, however the commutator $AB-BA$ is not contained in the radical of the Banach algebra generated by $A$ and $B$. 
Although this counterexample answers negatively to \cite[Open questions 3.7(2)]{BDFRZ} in the case of at least three-dimensional Banach lattices, there is a positive answer on two-dimensional Banach lattices.

% In \cite{BDFRZ} the authors proved that a positive commutator $AB-BA$  of positive compact operators $A$ and $B$ is quasi-nilpotent. They posed a question whether this is still true if only one of the operators $A$ and $B$ is considered to be compact. This problem was solved independently by R. Drnov\v sek \cite{roman} and N. Gao \cite{gao}. 

\section {Preliminaries}

Let $E$ be a Riesz space, and let $E^+$ denote the set of all positive vectors in $E$. 
A linear subspace $J$ of $E$ is said to be an {\it ideal} whenever $0\leq |x|\leq |y|$ and $y\in J$ imply $x\in J$. 
An order closed ideal is said to be a {\it band}.  The set 
$$ A^d:=\{x\in E:\; |x|\wedge |a|=0  \textrm{ for all }a\in A\}$$ is called the {\it disjoint complement} of a set $A$ in $E$. The disjoint complement of a nonempty set in $E$ is always a band in $E$. A band $B$ of $E$ is said to be a {\it projection band} if $E=B\oplus B^d.$% The corresponding band projection onto the band $B$ is denoted by $P_B.$
A Riesz space $E$ is said to be a {\it normed Riesz space} whenever $E$ is equipped with a lattice norm. If a normed Riesz space is a complete metric space with the metric induced by the norm, then it is called a {\it Banach lattice}. It is well-known that every normed Riesz space is Archimedean. 

A nonzero vector $a\in E^+$ is an {\it atom} in a normed Riesz space $E$ if $0 \le x, y \le a$ and $x \land y = 0$ imply 
either $x=0$ or $y=0$, or equivalently, if $0 \leq x \leq a$ implies $x=\lambda a$ for some $\lambda \geq 0$, i.e., 
the principal ideal $B_a$ generated by $a$ is one-dimensional. It turns out that $B_a$ is a projection band \cite{LuxZaa}.
If  every set of mutually disjoint nonzero vectors  of an Archimedean Riesz space $E$ is finite, then $E$ is a finite-dimensional atomic lattice that is order isomorphic to $\mathbb R^n$ (where $n=\dim E)$  with componentwise ordering.  If an Archimedean Riesz space $E$ contains only finitely many pairwise disjoint atoms, then $E$ is an order direct sum of the finite-dimensional atomic part of the lattice $E$ and its disjoint complement that is an atomless part of the lattice $E$. 

An operator $T$ on a Banach lattice $E$ is said to be positive whenever $T$ maps the positive cone of $E$ into itself. It is well-known that every positive operator on a Banach lattice is a bounded operator. By $\mathcal L(X)$ and $\mathcal L_+(E)$ we denote the Banach algebra of all bounded operators on a Banach space $X$ and the set of all positive operators on a Banach lattice $E$, respectively.  A family $\mathcal F$ of operators on a Banach space $X$ is {\it reducible} if there exists a nontrivial closed subspace of $X$ that is invariant under every operator from $\mathcal F$. If there exists a maximal chain $\mathcal C$ of closed subspaces of $X$ such that every subspace from the chain $\mathcal C$ is invariant under every operator from $\mathcal F$, then $\mathcal F$ is said to be {\it triangularizable}, and $\mathcal C$ is called a {\it triangularizing chain} for $\mathcal F$.  A family $\mathcal F$ of operators on a Banach lattice $E$ is said to be {\it ideal-reducible} if there exists a nontrivial closed ideal of $E$ that is invariant under every operator from $\mathcal F$. 
Otherwise, we say that $\mathcal F$ is {\it ideal-irreducible}.
A family  $\mathcal F$ of operators on a Banach lattice $E$ is said to be {\it ideal-triangularizable} whenever there exists a chain of closed ideals of $E$ that is maximal as a chain in the lattice of all closed ideals of $E$ and every ideal from the chain is invariant under $\mathcal F$.  Every ideal-triangularizable family of operators is also triangularizable \cite{Drn}. 
If $\mathcal C$ is a complete chain of closed subspaces of a Banach space $X$, the {\it predecessor} of $\mathcal M$ in $\mathcal C$ is denoted by $\mathcal M_-$. If every subspace $\mathcal M$ from a complete chain $\mathcal C$ of closed subspaces of $X$ is  invariant under the operator $A\in \mathcal L(X)$ and $\mathcal M\neq \mathcal M_-$, the induced operator $A_{\mathcal M}$ on $\mathcal M/\mathcal M_-$ is defined by $A_{\mathcal M}(x+\mathcal M_-)=Ax+\mathcal M_-$ for all $x\in \mathcal M.$
In the case of a finite-dimensional lattice $\mathbb R^n$ we will use more common terms decomposable, indecomposable and completely decomposable instead of ideal-reducible, ideal-irreducible and ideal-triangularizable, respectively. 
  For a more detailed treatment on triangularizability we refer the reader to \cite{RadRos}. 

Positive operators $A$ and $B$ on a Banach lattice {\it semi-commute} whenever $AB\geq BA$ or $AB\leq BA.$ The {\it super right-commutant} $[A\rangle$ and the {\it super left-commutant} $\langle A]$ of a positive operator $A$ on a Banach lattice $E$ are defined by
$$[A\rangle=\{B\in \mathcal L_+(E):\; BA\geq AB\} \qquad \textrm{and}\qquad \langle A]=\{B\in \mathcal L_+(E):\; AB\geq BA\},$$ respectively.
For the terminology and details not explained about Banach lattices and operators on them we refer the reader to \cite{AlAp} and \cite{Sch}.

The {\it radical} of a Banach algebra $\mathcal A$ is the intersection of all maximal modular left ideals of $\mathcal A$ which coincides with the intersection of all maximal modular right ideals of $\mathcal A$. If $\mathcal A$ is a unital Banach algebra, then the radical of $\mathcal A$ is equal to the set 
$$\{a\in \mathcal A:\; r(ba)=0 \textrm{ for all }b\in \mathcal A\}=\{a\in \mathcal A:\; r(ab)=0 \textrm{ for all }b\in \mathcal A\}$$
where $r(x)$ denotes the spectral radius of an element $x\in \mathcal A$. 
Therefore, the radical of a unital Banach algebra is the largest among all ideals that consist of quasi-nilpotent elements. If a Banach algebra $\mathcal A$ is not unital, then we consider the spectrum of $a\in \mathcal A$ as the spectrum of the element  $a$ in the standard unitization of the algebra $\mathcal A$. 
For the terminology not explained throughout the text regarding Banach algebras we refer the reader to \cite{BonDun}.

\section {Counterexamples}\label{counter}

In \cite{BDFRZ} the authors proved \cite[Theorem 2.3]{BDFRZ} that the commutator $AB-BA$ of real matrices $A$ and $B$ is a nilpotent matrix whenever $AB\geq BA\geq 0$.  
The authors also extended this result to the case when $A\in \mathcal C^p$ and $B\in \mathcal C^q$ are bounded operators on $l^2$ with $AB\geq BA\geq 0$ where $1\leq p,q\leq \infty$ and $1/p+1/q=1$.
In this case, the commutator is a quasi-nilpotent trace class operator on $l^2$. See \cite[Theorem 2.4]{BDFRZ}. They asked (\cite[Open questions 3.7(2)]{BDFRZ} if the commutator $AB-BA$ is contained in the radical of the Banach algebra generated by $A$ and $B$. 
In this section we provide a counterexample on every at least three-dimensional Banach lattice. 

For the construction of our counterexample to \cite[Open questions 3.7(2)]{BDFRZ} we need some preparation. The proof of the following result can be found in \cite[Theorem 26.10]{LuxZaa}. We provide our proof for the sake of completness.  

\begin {proposition}\label{nevem}
Every infinite-dimensional Banach lattice has at least countably infinitely many nonzero pairwise disjoint positive vectors. 
\end {proposition}

\begin {proof}
Let us choose an aribtrary infinite-dimensional Banach lattice $E$. 
Suppose first that $E$ does not contain atoms. 
We claim that there exists an increasing sequence of sets $\{\mathcal A_n\}_{n\in\mathbb N}$ of positive nonzero pairwise disjoint vectors  with $|\mathcal A_n|=n$ and a sequence of positive nonzero vectors $\{y_n\}_{n\in\mathbb N}$ such that 
$y_n$ is disjoint with $\mathcal A_n$ for every positive integer $n.$
Let us validate the statement above when $n=1$. If $y_0$  is an arbitrary nonzero positive vector in $E$, then the fact that $E$ is without atoms implies that there exist
nonzero positive vectors $0\leq x_1,y_1\leq y_0$ satisfying $x_1\wedge y_1=0.$
We choose $\mathcal A_1=\{x_1\}$ and $y_1$ is disjoint with $\mathcal A_1.$
Suppose that there exist  sets $\mathcal A_1\subseteq \cdots\subseteq \mathcal A_{n}$ of positive nonzero pairwise disjoint vectors with $|\mathcal A_j|=j$ for $1\leq j\leq n$, and positive nonzero vectors $y_1,\ldots, y_n$ such that 
$y_j\in  \mathcal A_j^d$ for $1\leq j\leq n$.
Since $y_n$ is not atom, there exist positive nonzero vectors $x_{n+1}$ and $y_{n+1}$ satisfying $0\leq x_{n+1},y_{n+1}\leq y_n$ and $x_{n+1}\wedge y_{n+1}=0$. 
It is obvious that the vectors $x_{n+1}$ and $y_{n+1}$ are disjoint with $\mathcal A_n$. The set
$$\mathcal A_{n+1}:=\mathcal A_n\cup \{x_{n+1}\}$$ obviously satisfies $\mathcal A_n\subseteq \mathcal A_{n+1}$, $|\mathcal A_{n+1}|=n+1$ and $y_{n+1}\in A_{n+1}^d.$

Let $\mathcal A=\bigcup_{n=1}^\infty\mathcal A_n$.  The set $\mathcal A$ is countably infinite, and arbitrary two vectors from $\mathcal A$ are disjoint. 

Suppose now that $E$ contains at least one atom. By Zorn's lemma there exists a maximal set $\mathcal A$ of pairwise disjoint atoms of norm one. 
If $\mathcal A$ is infinite, then we are done. Suppose otherwise that $\mathcal A$ is a finite set, so that we have $E=\mathcal A^{dd}\oplus \mathcal A^d$ where 
$\mathcal A^{dd}$ is a finite-dimensional atomic lattice and $\mathcal A^d$ is an infinite-dimensional  atomless lattice. 
By already proved, $\mathcal A^d$ contains at least countably infinitely many nonzero positive pairwise disjoint vectors. 
\end {proof}

\begin {lemma}\label{funkcionali}
Let $E$ be a Banach lattice of dimension at least $n$ and let $x_1,\ldots, x_n$ be nonzero positive pairwise disjoint vectors. 
Then there exist nonzero positive functionals $\varphi_1,\ldots, \varphi_n$ on $E$ such that $\varphi_i(x_j)=\delta_{ij}$ for $i,j=1,\ldots, n.$
\end {lemma}

\begin {proof}
Let us choose an arbitrary $1\leq i\leq n$ and let  $\mathcal J_i$ be the ideal generated by the vector $x_i$. 
By \cite [Theorem 39.3.]{Zaa}, there exists a positive functional $\psi_i$ on $E$ with $\psi_i(x_i)=1.$ 
Let $\psi_i|_{\mathcal J_i}$ be the restriction of the functional $\psi_i$ on the ideal $\mathcal J_i$. By \cite[Theorem 20.5.]{Zaa} and its following remark there exists a positive functional $\varphi_i$ on $E$ which extends $\psi_i|_{\mathcal J_i}$ and $\varphi_i$ is zero on $\mathcal J_i^d.$ This implies $\varphi_i(x_j)=\delta_{ij}$ for all $j=1,\ldots,n$.
\end {proof}

Although Lemma \ref{funkcionali} is cruicial for Example \ref{Example 1} and Example \ref{Example 2}, it is also of its own interest.  
Before we indicate its importance let us recall first the well-known Jacobson's density theorem for Banach algebras \cite{BonDun}. 

\begin {theorem}
Let $\mathcal A$ be a Banach algebra and let $\pi:\mathcal A\to \mathcal L(X)$ be a continuous irreducible representation on a Banach space $X$. If $x_1,\ldots,x_n$ are linearly independent in $X$ and $y_1,\ldots, y_n$ are arbitrary in $X$, then there exists $a\in \mathcal A$ such that $\pi(a)x_i=y_i$ for all $i=1,\ldots, n.$
\end {theorem}

In the special case when $\pi$ is the identity representation on $\mathcal L(X)$ Jacobson's density theorem states that a finite set of linearly independent vectors can be mapped onto arbitrary vectors by bounded operators. 
This fact can be also proved with an application of the Hahn-Banach theorem without applying the Jacobson's density theorem. Now it is natural to ask whether arbitrary linearly independent positive vectors in Banach lattices can be mapped to arbitrary positive vectors by positive operators. The following example shows that in general it cannot be done. 

\begin {example}\label{exam1}
The space $\mathbb R^2$ with a canonical ordering and equipped with the norm $\|\cdot\|_2$ is a Banach lattice. The only matrix mapping 
$\left[\begin {array}{c}
1\\
0 \end {array}\right]$ to 
$\left[\begin {array}{c}
1\\
0 \end {array}\right]$ and 
$\left[\begin {array}{c}
1\\
1 \end {array}\right]$ to 
$\left[\begin {array}{c}
0\\
1 \end {array}\right]$ is 
$$\left[\begin {array}{cc}
1&-1\\
0& 1\\
\end {array}\right]$$ which is clearly not positive. 
\end {example}

However, for a finite set of nonzero pairwise disjoint positive vectors we have a positive result. 

\begin {theorem}\label{jac}
Let $E$ be a Banach lattice and let $x_1,\ldots, x_n$ be arbitrary nonzero positive pairwise disjoint vectors in $E$. Then for arbitrary positive vectors $y_1,\ldots, y_n$ in $E$ there exists a positive operator $T$ on $E$ such that $Tx_j=y_j$ for all $j=1,\ldots, n.$
\end {theorem}

\begin {proof}
By Lemma \ref{funkcionali} there exist positive functionals $\varphi_1,\ldots, \varphi_n$ on $E$ such that $\varphi_i(x_j)=\delta_{ij}$ for all $i,j=1,\ldots, n.$
The positive operator 
$$T=\sum_{i=1}^n y_i\otimes \varphi_i$$ satisfies 
$$Tx_j=\sum_{i=1}^n \varphi_i(x_j)y_i=\sum_{i=1}^n\delta_{ij}y_j=y_j$$ for all $j=1,\ldots,n.$
\end {proof}

Throughout the rest of this section $E$ is assumed to be at least three-dimensional Banach lattice. 
Let the vectors $x_1,x_2$ and $x_3$ and the functionals $\varphi_1,\varphi_2$ and $\varphi_3$ be as in Lemma \ref{funkcionali}. 
The vectors $x_1,x_2$ and $x_3$, in fact,  exist since either $E$ is finite-dimensional and so it is order isomorphic to $\mathbb R^n$ with the componentwise order, or  $E$ is infinite-dimensional and  we may apply  Proposition \ref{nevem}. 

The following example shows that the commutator $AB-BA$ of compact operators $A$ and $B$ with $AB\geq BA\geq 0$ is in general  not contained in the radical of the Banach algebra generated by $A$ and $B$ if we do not assume $A\geq 0$. 

\begin {example}\label{Example 1}
An easy calculation shows that operators 
$$A=(-x_2+x_3)\otimes \varphi_1 +(x_1+x_2)\otimes \varphi_2 \qquad \textrm{and}\qquad B=x_2\otimes (\varphi_2+\varphi_3)$$ satisfy
$$AB=(x_1+x_2)\otimes (\varphi_2+\varphi_3) \qquad \textrm{and} \qquad BA=x_2\otimes \varphi_2,$$ so that $AB$ and $BA$ are positive operators on $E$. 
Since
$$AB-BA=x_1\otimes(\varphi_2+\varphi_3)+x_2\otimes \varphi_3\geq 0,$$ we also have 
$AB\geq BA\geq 0.$
The commutator $AB-BA$ is not contained in the radical of the Banach algebra $\mathcal A$ generated by the operators $A$ and $B$ since the operator 
$$A(AB-BA)=(x_3-x_2)\otimes \varphi_2+(x_1+x_3)\otimes \varphi_3$$ has $1$ as its eigenvalue with a corresponding eigenvector $x_1+x_3.$ 
%Therefore, 
%$r(A(AB-BA))\neq 0$ and $AB-BA$ is not contained in the radical of the Banach algebra generated by $A$ and $B$. 
\end {example}

%By taking adjoints in Example \ref{Example 1} we see that Question \ref{quest1} and Question \ref{quest2} 
%do not hold for dual Banach lattices of dimension at least three even when $A\geq 0.$ 

The following example shows that the commutator $AB-BA$ of compact operators $A$ and $B$ with $AB\geq BA\geq 0$ is in general  not contained in the radical of the Banach algebra generated by $A$ and $B$ if we do not assume $B\geq 0$.

\begin {example}\label{Example 2}
Similarly as in Example \ref{Example 1} we can see that operators  
$$A=(x_2+x_3)\otimes \varphi_2\qquad \textrm{and}\qquad B=x_2\otimes \varphi_1+(x_2-x_1)\otimes \varphi_2+x_1\otimes \varphi_3$$ satisfy
$AB\geq BA\geq 0.$ An easy calculation also shows that 
$(AB-BA)B(x_2+2x_3)=x_2+2x_3$, so that $1$ is in the spectrum of the operator $(AB-BA)B.$
%The well-known formula for spectral radius gives us 
%$$r(B(AB-BA))=r((AB-BA)B)\geq 1,$$ which implies that $AB-BA$ is not in the radical of the Banach algebra generated by $A$ and $B$. 
\end {example}

\section {Results}

Although the counterexamples in Section \ref{counter} answer negatively to \cite[Open questions 3.7(2)]{BDFRZ} in the case of at least three-dimensional lattices, we now provide a positive result in the two-dimensional case. 

\begin {proposition}\label{dvorazsezno}
Suppose that matrices  $A$ and $B$ in $M_2(\mathbb R)$ satisfy $AB\geq BA\geq 0.$ Then either  $AB=BA$ or $A$ and $B$ are simultaneously completely decomposable. 
\end {proposition}

If two matrices commute, then it is well-known that $A$ and $B$ are simultaneously triangularizable. Therefore we can conclude that real matrices $A,B\in M_2(\mathbb R)$ are simultaneously triangularizable whenever they satisfy the conditions of the preceding proposition. 
It should be clear that the same result does not hold for $n\geq 3$. If matrices $A$ and $B$ are simultaneously triangularizable, then $A$ and $B$ are upper-triangular with respect to some basis of the space $\mathbb C^n$. Now it is not difficult to see that the matrix $A(AB-BA)$ is nilpotent. However, Example \ref{Example 1} shows that there exist matrices $A,B\in M_3(\mathbb R)$ satisfying $AB\geq BA\geq 0$, and the matrix $A(AB-BA)$ is not nilpotent.

\begin {proof}
We may assume $B\neq 0$, as otherwise $A$ and $B$ commute. 
Suppose first that the matrix $AB$ is indecomposable. Then  \cite[Lemma 8.14]{AbraAlip},  
 \cite[Corollary 8.23]{AbraAlip} and the fact that we have $r(AB)=r(BA)$ imply $AB=BA$. 

Assume now that $AB$ is decomposable. The standard subspace invariant under $AB$ is invariant under $BA$ as well, since $AB\geq BA\geq 0$. The fact that the space is two-dimensional implies that $AB$ and $BA$ are simultaneously similar to upper-triangular matrices and ¸that this similarity can be obtained by a permutation matrix. 
We may obviously assume that $AB$ and $BA$ are already upper-triangular. 
Since $AB\geq BA\geq 0$ and the spectra of upper-triangular matrices are contained on their diagonals, we see that the matrices $AB$ and $BA$ have the same diagonal. Let us write
$$A=\left[\begin {array}{cc}
a&b\\
c&d\end {array}\right] \qquad \textrm{and}\qquad B=\left[\begin {array}{cc}
e&f\\
g&h\\
\end {array}\right].$$
From $$AB=\left[\begin {array}{cc}
\alpha&\gamma\\
0&\beta\\
\end {array}\right] \qquad \textrm{and}\qquad 
BA=\left[\begin {array}{cc}
\alpha&\delta\\
0&\beta\\
\end {array}\right]$$ we obtain 
$\alpha=ae+bg=ae+cf$, $\beta=cf+dh=bg+dh$ and $ce+dg=ag+ch=0.$

Let us first consider the case $g=0$. Then $ce=cf=ch=0.$
If $c=0$, then $A$ and $B$ are upper-triangular. If $c\neq 0$, then $e=f=h=0$ and $B=0$ which is impossible due to our assumption.   

Assume now that $g\neq 0$. Then we have 
$b=\frac{fc}{g}$, $d=-\frac{ce}{g}$ and $a=-\frac{ch}{g}.$ Now, a direct calculation shows
$\gamma=af+bh=0$ and $\delta=eb+fd=0$. Therefore $AB=BA$. 
\end {proof}

\begin {corollary}
Suppose that matrices  $A$ and $B$ in $M_2(\mathbb R)$ satisfy $AB\geq BA\geq 0.$
Then $AB-BA$ is in the radical of the Banach algebra generated by $A$ and $B$. 
\end {corollary}

\begin {proof}
If $AB=BA$, then there is nothing to prove. 
Otherwise, Proposition \ref{dvorazsezno} implies that 
$A$ and $B$ are simultaneously completely decomposable. By McCoy's theorem \cite[Theorem 1.3.4]{RadRos} the matrix $p(A,B)(AB-BA)$ is nilpotent for every polynomial $p$ in two non-commuting variables. Since the spectral radius is continuous on the set of all $n\times n$ matrices, a matrix $C(AB-BA)$ is nilpotent for every $C$ in the Banach algebra generated by $A$ and $B$ which finishes the proof. 
\end {proof}
%\begin {lemma}
%Let $A$ and $B$ be positive power-compact operators on a Banach lattice $E$ with a positive commutator $AB-BA$. Then for arbitrary non-commutative polynomial $p$ in two variables without a constant term, the operator $p(A,B)$ is a power-compact operator. 
%\end {lemma}

Although real matrices $A$ and $B$ satisfying $AB\geq BA\geq 0$ are not triangularizable in general, the situation becomes completely different if $A$ and $B$ are assumed to be positive. 

\begin {proposition}\label{spek izrek}
Let $A$ and $B$ be positive power-compact operators on a complex Banach lattice $E$ with a positive commutator $AB-BA$. Then $A$ and $B$ are simultaneously triangularizable. 
\end {proposition}

\begin {proof}
Assume that $A$ and $B$ are power-compact. Since a closed subspace is invariant under $A$ and $B$ if and only if it is invariant under $A$ and $A+B$, by \cite[Lemma 2.2]{Dkankom} we may assume that $0\leq A\leq B$ as we can replace $B$ with $A+B$.
We claim that $A$ and $B$ are simultaneously triangularizable. 
Let $\mathcal C$ be a maximal chain of closed ideals that are invariant under $B$. Since $0\leq A\leq B$, every ideal in $\mathcal C$ is also invariant under $A$. 
For every $\mathcal J\in \mathcal C$ the power-compact operator $B_{\mathcal J}$ is ideal-irreducible. \cite[Theorem 1.3]{Dkankom} and \cite[Corollary 3.3]{Dkankom} imply that 
$A_{\mathcal J}$ and $B_{\mathcal J}$ commute on $\mathcal J/\mathcal J_-.$
Let $\mathcal C'$ be a maximal chain of closed subspaces invariant under $A$ and $B$ which contains the chain $\mathcal C$. %It should be noted that such chains exist by Zorn's lemma.
 
We claim that $\mathcal C'$ is a triangularizing chain for operators $A$ and $B$. 
Since the chain $\mathcal C'$ is maximal as a chain of subspaces invariant under $A$ and $B$, (i) and (ii) in \cite[Theorem 7.1.9]{RadRos} are satisfied. Therefore we only need to prove that the dimension of  $\mathcal M/\mathcal M_-$ is at most one  for an arbitrary subspace $\mathcal M$ in the chain $\mathcal C'$. 

Choose an arbitrary subspace $\mathcal M\in \mathcal C'.$
We claim that there exists a closed ideal $\mathcal J$ in $\mathcal C$ satisfying $\mathcal J_-\subseteq \mathcal M_-\subseteq \mathcal M\subseteq \mathcal J$ where $\mathcal J_-$ and $\mathcal M_-$ are predecessors of $\mathcal J$ and $\mathcal M$ in chains $\mathcal C$ and $\mathcal C'$, respectively. 
Let $\mathcal J$ be the intersection of all closed ideals in $\mathcal C$ that contain $\mathcal M$. Let $\mathcal J_-$ and $\mathcal M_-$ be the predecessors of $\mathcal J$ and $\mathcal M$ in the chains $\mathcal C$ and $\mathcal C'$, respectively.
Assume first that $\mathcal J_-=\mathcal J.$ If $\mathcal I$ is an arbitrary closed ideal in $\mathcal C$ that is properly contained in $\mathcal J$, then $\mathcal I$ is contained in $\mathcal M$, so that $\mathcal J_-$ is also contained in $\mathcal M$. 
Since $\mathcal M\subseteq \mathcal J$ and $\mathcal J_-\subseteq \mathcal M$, we have $\mathcal J=\mathcal J_-=\mathcal M=\mathcal M_-$.  
Assume now that $\mathcal J_-$ is a proper subset of $\mathcal J.$ Since $\mathcal C'$ is a chain, then $\mathcal J_-$ is a subset of $\mathcal M$, as $\mathcal J$ is the intersection of all closed ideals of $\mathcal C$ that contain $\mathcal M$, and $\mathcal J_-$ is a closed linear span of all closed ideals of $\mathcal C$ that are contained in $\mathcal M$. If $\mathcal J_-=\mathcal M$, then $\mathcal J_-$ is actually equal to $\mathcal J$ which is impossible due to our assumption $\mathcal J_-\neq \mathcal J$. Therefore $\mathcal J_-$ is a proper subset of $\mathcal M$ which implies $\mathcal J_-\subseteq \mathcal M_-$ so that we have $\mathcal J_-\subseteq \mathcal M_-\subseteq \mathcal M\subseteq \mathcal J$ and the claim is proved. 

Assume now that the dimension of $\mathcal M/\mathcal M_-$ is at least $2$. Since $A_{\mathcal J}$ and $B_{\mathcal J}$ commute on $\mathcal J/\mathcal J_-$ and $\mathcal J_-\subseteq \mathcal M_-$, 
the induced operators $A_{\mathcal M}$ and $ B_{\mathcal M}$ of the operators $A$ and $B$ on the quotient Banach space $\mathcal M/\mathcal M_-$ also commute. 
Suppose first that at least one of $A_{\mathcal M}$ and $B_{\mathcal M}$ is  not a scalar multiple of the identity operator. Without any loss of generality assume that $A_{\mathcal M}$ is not a scalar multiple of the identity operator. 
If $A_{\mathcal M}$ is nilpotent, then its kernel is invariant under both $A_{\mathcal M}$ and $B_{\mathcal M}$, so that the preimage of the kernel of the operator $A_{\mathcal M}$ is  a closed subspace invariant under $A$ and $B$ that is  properly contained between $\mathcal M_-$ and $\mathcal M$. This contradicts the maximality of $\mathcal C'$. If $A_{\mathcal M}$ is not nilpotent, then some power $(A_{\mathcal M})^n$ is a nonzero compact operator. By Lomonosov's theorem \cite{Lom} there exists a nontrivial closed subspace of $\mathcal M/\mathcal M_-$ invariant under $A_\mathcal M$ and $B_\mathcal M$. Similarly as above this is again a contradiction with maximality of $\mathcal C'$. If both $A_{\mathcal M}$ and $B_{\mathcal M}$ are scalar multiples of the identity operator, then every subspace of $\mathcal M/\mathcal M_-$ is invariant under $A_{\mathcal M}$ and $B_{\mathcal M}$ which would (similarly as above) lead to a contradiction with the maximality of the chain $\mathcal C'$.

Therefore the dimension of $\mathcal M/\mathcal M_-$ is less than or equal to one and the chain $\mathcal C'$ is a triangularizing chain for the operators $A$ and $B$. 
%
%Given any subspace $\mathcal M$ in $\mathcal C'$, the proof above actually shows that the induced operators $\widehat A$ and $\widehat B$ commute on $\mathcal M/\mathcal M_-$.  
%This implies that for an arbitrary non-commuting polynomial $p$ in two variables with non-negative coefficients the diagonal coefficients of the operator $p(A,B)(.$
%
% 
%Operators $A_{\mathcal J}$ and $B_{\mathcal J}$ are triangularizable on $\mathcal J/\mathcal J_-$ by \cite[Corollary 2.6]{konva} which implies that the chain $\mathcal C$ is contained in a triangularizing chain $\mathcal C'$ for operators $A$ and $B$. 
%
%To see the additional statement, apply Corollary \ref{cor to main 1} and proceed as in the proof of Ringrose's theorem \cite[Theorem 7.2.3]{RadRos}.
\end {proof}

In the proof of the preceding Proposition we applied the Lomonosov theorem to prove that the chain $\mathcal C'$ is a triangularizing chain for operators $A$ and $B$. 
In the real case this cannot be done always. If the dimension of $\mathcal M/\mathcal M_-$ is finite and at least three, then we can proceed as in the proof of Proposition \ref{spek izrek} also in the real case, since every real matrix on $\mathbb R^n$ with $n$ at least three has a nontrivial hyperinvariant subspace, which leads to a contradiction with the maximality of the chain $\mathcal C'.$ 
However, the problem occurs  when the dimension of $\mathcal M/\mathcal M_-$ is precisely two.    
In the following example we construct a positive indecomposable matrix $A\in \mathbb R^3$ with a two dimensional invariant subspace $\mathcal M$ such that the restriction of $A$ to $\mathcal M$ is irreducible. 

\begin {example}
Take the indecomposable matrix $A=\left[\begin {array}{ccc}
0&1&0\\
0&0&1\\
1&0&0\\
\end {array}\right]$ and 
$\mathcal M$ the vector space spanned by $\left[\begin{array}{c}
2\\
-1\\
-1\\ 
\end {array}\right]$ and 
$\left[\begin{array}{c}
0\\
1\\
-1\\ 
\end {array}\right]$. A straightforward calculations show that subspace $\mathcal M$ is invariant under $A$, and the restriction of $A$ to $\mathcal M$ is irreducible, since the only (real) eigenspace of $A$ is spanned by the vector $\left[\begin{array}{c}
1\\
1\\
1\\
\end{array}\right]$ which does not belong to $\mathcal{M}$.  
\end {example}

It should be noted that every positive matrix on $\mathbb R^2$ is triangularizable since the spectral radius of a positive matrix is always in its spectrum. 

The following theorem is the main result of this paper. It provides a positive answer to the remaining part of \cite[Open questions 3.7(1)]{BDFRZ} and at the same time it tells us actually that the commutator $AB-BA$ of a positive compact operator and a positive operator is contained in the radical of a bigger Banach algebra which contains the Banach algebra generated by $A$ and $B$. 

\begin {theorem}\label{keksington}
Let $A$ be a  positive compact operator on a Banach lattice $E$ and let $\mathcal A_1$ and $\mathcal A_2$ be the Banach algebras generated by $\langle A]$ and $[A\rangle$, respectively. 

\begin {enumerate}
\item [(a)] The operator $AB-BA$ is in the radical of the Banach algebra $\mathcal A_1$ for every $B\in \mathcal A_1$.
\item [(b)] The operator $AB-BA$ is in the radical of the Banach algebra $\mathcal A_2$ for every $B\in \mathcal A_2$. 
%\item [(b)] If $A$ and $B$ are both power-compact, then some power of $AB-BA$ is in the radical of $\mathcal A$. 
\end {enumerate}
\end {theorem}

\begin {proof}
Let $\mathcal A$ be the algebra generated by $\langle A]$ and let us choose $B$ and $C$ from $\mathcal A$. 
The operators $B$ and $C$ are finite linear combinations of words in letters from $\langle A]$. Let $B_1,\ldots, B_n$ and $C_1,\ldots, C_m$ be operators in $\langle A]$ which occur as letters in words whose appropriate linear combinations are equal to $B$ and $C$, respectively. The operator $D:=A+B_1+\cdots+B_n+C_1+\cdots +C_m$ satisfies $0\leq A\leq D$ and the commutator 
$$AD-DA=\sum_{i=1}^n(AB_i-B_iA)+\sum_{j=1}^m(AC_j-C_jA)$$ is obviously a positive operator on $E$.  
Let $\mathcal C$ be a maximal chain of closed ideals invariant under $D$. Maximality of the chain $\mathcal C$ implies that $\mathcal C$ is complete and that for every $\mathcal J\in \mathcal C$ with $\dim (\mathcal J/\mathcal J_-)>1$ the operator $D_{\mathcal J}$ is ideal-irreducible on $\mathcal J/\mathcal J_-$. Also, every closed ideal in $\mathcal C$ is invariant under $A$, $B_i$ and $C_j$ for all $1\leq i\leq n$ and $1\leq j\leq m$, so that it is also invariant under $A$, $B$ and $C$. 
\cite[Corollary 3.5]{Gao} implies that $A_{\mathcal J}$ and $D_{\mathcal J}$ commute on $\mathcal J/\mathcal J_-$. 
This implies, in particular, that $A_\mathcal J$ commutes with $(B_i)_{\mathcal J}$ for all $1\leq i\leq n$, so that actually $A_{\mathcal J}$ and $B_{\mathcal J}$ commute on $\mathcal J/\mathcal J_-.$
%Choose an arbitrary polynomial $p$ in two non-commuting variables. 
By Ringrose's theorem for complete chains \cite[Theorem 7.2.7]{RadRos}, it follows that 
$$\sigma(C(AB-BA))=\bigcup_{{\tiny \begin {array}{c}
						\mathcal J\in \mathcal C\\
						\mathcal J\neq \mathcal J_-\end {array}}}\sigma (C_\mathcal J(A_{\mathcal J}B_{\mathcal J}-B_{\mathcal J}A_{\mathcal J}))\cup \{0\}=\{0\}.$$

%For every $\mathcal J\in \mathcal C$, let $\mathcal C_\mathcal J$ be the maximal chain of closed ideals invariant under the zero operator $p(A_{\mathcal J},B_{\mathcal J})(A_{\mathcal J}B_{\mathcal J}-B_{\mathcal J}A_{\mathcal J}).$
%Gao proved that the chain 
%$$\mathcal C'=\bigcup_{\mathcal J\in \mathcal C}\mathcal C_{\mathcal J}$$ is actually an ideal-triangularizing chain for the compact operator $p(A,B)(AB-BA).$

If $B$ and $C$ in $\mathcal A_1$ are arbitrary, then there exist  sequences of operators $\{B_n\}_{n\in\mathbb N}$ and $\{C_n\}_{n\in\mathbb N}$ in $\mathcal A$ converging to $B$ and $C$, respectively, so that 
$C_n(AB_n-B_nA)$ converges to $C(AB-BA).$ Since the norm limit of compact quasi-nilpotent operators is  quasi-nilpotent, $C(AB-BA)$ is quasi-nilpotent which finishes the proof of (a). 
The proof of (b) is similar so we omit it.

\end {proof}

\begin {corollary}\label{mlekozavse}
Suppose that $A$ and $B$ be positive operators on a Banach lattice $E$ with a positive commutator $AB-BA$. If at least one of the operators $A$ and $B$ is compact, then $AB-BA$ is contained in the radical of the Banach algebra generated by $A$ and $B$.   
\end {corollary}

\begin {proof}
Suppose that $A$ is a compact operator. Let us denote by $\mathcal A$ the Banach algebra generated by $A$ and $B$, and let us denote by $\mathcal A_1$ the Banach algebra generated by $\langle A]$. 
Since $AB-BA$ is in the radical of $\mathcal A_1$ by Theorem \ref{keksington}(a), it follows that $r(C(AB-BA))=0$ for an arbitrary operator $C\in \mathcal A_1$. 
The fact that we have $\mathcal A\subseteq \mathcal A_1$ implies $r(C(AB-BA))=0$ for every $C\in \mathcal A$ which means that $AB-BA$ is in the radical of $\mathcal A$.
For the proof when $B$ is compact we repeat the proof above and apply Theorem \ref{keksington}(b).  
\end {proof}

In the case of power-compact operators we can obtain the following result. 

\begin {theorem}\label{extensione}
If $A$ and $B$ are positive  power-compact operators on a Banach lattice with a positive commutator $AB-BA$, then some power of $AB-BA$ is in the radical of the Banach algebra generated by $A$ and $B$. 
\end {theorem}

\begin {proof}
Let $\mathcal A$ be the Banach algebra generated by $A$ and $B$. By \cite[Lemma 2.3]{Dkankom}, the operator $AB-BA$ is a power-compact operator. If $AB-BA$ is nilpotent, then it is clear that the apropriate power of the operator $AB-BA$ is in the radical of $\mathcal A$. 
Thefore we may assume that $AB-BA$ is not nilpotent.  There exists a positive integer $n$ such that $(AB-BA)^n$ is a compact operator.

Suppose first that the lattice $E$ is a complex Banach lattice. Proposition \ref{spek izrek} implies that $A$ and $B$ are simultaneously triangularizable. 
Let $\mathcal C$ be one of the triangularizing chains for operators $A$ and $B$. It is not hard to see that $\mathcal C$ is a triangularizing chain for every operator $C$ from $\mathcal A$. Since every diagonal coefficient of the operator $AB-BA$ is zero, every diagonal coefficient of the operator $C(AB-BA)^n$ is zero as well. 
Compactness of the operator $(AB-BA)^n$ and \cite[Theorem 7.2.3]{RadRos} implies that $C(AB-BA)^n$ is quasi-nilpotent, and since $C\in\mathcal A$ was arbitrary we obtain that $(AB-BA)^n$ is contained in the radical of $\mathcal A$. 

Assume now that $E$ is a real Banach lattice. Let $E_{\mathbb C}$ be the complexification of the Banach lattice $E$ and let $A_{\mathbb C}$ and $B_{\mathbb C}$ be the standard complexifications of operators $A$ and $B$ on $E_{\mathbb C}$, respectively.   It is not hard to see that we have $A_{\mathbb C}B_{\mathbb C}\geq B_{\mathbb C}A_{\mathbb C}$ on $E_{\mathbb C}$, so that by Proposition \ref{spek izrek} operators $A_{\mathbb C}$ and $B_{\mathbb C}$ are simultaneously triangularizable.
Since we have
$$(A_{\mathbb C}B_{\mathbb C}-B_{\mathbb C}A_{\mathbb C})^n=((AB-BA)_{\mathbb C})^n=((AB-BA)^n)_{\mathbb C},$$
the operator $(A_{\mathbb C}B_{\mathbb C}-B_{\mathbb C}A_{\mathbb C})^n$ is a nonzero compact operator on $E_{\mathbb C}$. 
By the proof above  we see that $(A_{\mathbb C}B_{\mathbb C}-B_{\mathbb C}A_{\mathbb C})^n$ is contained in the radical of the algebra $\mathcal A_{\mathbb C}$ which implies that $r(C_{\mathbb C}(A_{\mathbb C}B_{\mathbb C}-B_{\mathbb C}A_{\mathbb C})^n)=0$ for all $C_{\mathbb C}\in \mathcal A_{\mathbb C}.$
This implies that $r(C(AB-BA)^n)=0$ for all $C\in \mathcal A$ finishing the proof. 
\end {proof}

\section {A result for the Volterra operator and the Donoghue  operator}

An operator on a Banach space is said to be {\it unicellular} if its lattice of closed invariant subspaces is totally ordered. 
An application of the Lomonosov theorem shows that every compact operator is triangularizable which implies that every unicellular compact operator has a unique triangularizing chain. 
The Volterra operator acting on $L^2[0,1]$ defined by $(Vf)(x)=\int_0^x f(t)dt$ is a well-known example of a unicellular operator \cite{RadRos2} with its lattice of closed invariant subspaces consisting of  subspaces of the form
$L^2[t,1]$, i.e., of all functions from $L^2[0,1]$ which are zero almost everywhere on the interval $[0,t].$ Volterra operator has a lot of interesting properties. For an example, 
if a normal operator $N$ commutes with a Volterra operator, then $N$ is a multiple of the identity operator \cite{ZemVolt}. An operator $T$ on a Banach space $X$ is said to be a {\it strong quasi-affinity} whenever for arbitrary closed subspace $\mathcal M$ invariant under $T$ we have that  $T\mathcal M$ is dense in $\mathcal M$. 

The Volterra operator has this remarkable property. It is not hard to see that the range of $V$ is dense in $L^2[0,1].$ 
The function $x\mapsto \frac{x-t}{1-t}$ induces a topological isomorphism $\Phi:L^2[0,1]\to L^2[t,1]$ defined by 
$$(\Phi f)(x)=f\left(\frac{x-t}{1-t}\right).$$
Let us denote by $V_t$ the restriction of the Voterra operator to its invariant subspace $L^2[t,1].$
It is not hard to see that $V_t$ satisfies 
$$(V_t\Phi f)(s)=(1-t)\int_0^\frac{s-t}{1-t}f(v)dv=(1-t)(\Phi V f)(s)$$ for all $s\in [t,1].$ 
Since $V$ has a dense range, $V_t$ also has a dense range. 

The other example of a unicellular operator is the Donoghue operator on a separable Hilbert space $\mathcal H$. 
Let $\{e_n\}_{n\in\mathbb N_0}$ be an orthonormal basis of $\mathcal H$ and let $\{w_n\}_{n\in\mathbb N}$ be a sequence of nonzero complex numbers such that $\{|w_n|\}_{n\in \mathbb N}$ is monotone decreasing and is in the space $l^2.$ The {\it Donoghue operator} $S$ with a weight sequence $\{w_n\}_{n\in\mathbb N}$ is defined by $Se_0=0$ and $Se_n=w_ne_{n-1}$ for $n>0.$  The operator $S$ is a unicellular operator \cite{RadRos2} with its lattice of closed invariant subspaces equal to 
$\{\mathcal M_n\}_{n\in \mathbb N_0}$ where $\mathcal M_n$ is the linear span of the vectors $e_0,\ldots, e_n.$ An easy calculation shows that $S\mathcal M_n=\mathcal M_{n-1}$, so that the induced operator by the operator $S$ on the quotient Banach space $\mathcal M_n/\mathcal M_{m}$ is nonzero whenever the dimension of $\mathcal M_n/\mathcal M_m$ is at least $2$. The later fact will be considered in general in Proposition \ref{zdravje}. We will be only concerned with the positive (the weights are positive) Donoghue operator acting on $l^2$ ordered componentwise. 

It is well-known that the Volterra operator and the Donoghue operator are compact operators. 

%\begin {proof}
%Let us prove first that the range of the operator $V$ is dense in $L^2[0,1].$ 
%It is clear that the range of $V$ contains every polynomial without the constant term. 
%Choose an arbitrary $\epsilon>0$ and a function $f\in L^2[0,1]$. There exists a continuous function $g$ on $[0,1]$ such that $\|f-g\|_2<\epsilon.$
%If $g(0)=0$, then there exists a polynomial $p$ without a constant term such that $\|p-g\|_\infty<\epsilon$. Since the $L^2$-norm of a continuous function is smaller of its supremum norm, we have $\|f-p\|_2<2\epsilon.$
%If $g(0)\neq 0$, then there exists a continuous function $h$ on $[0,1]$ with $h(0)=0$ and $\|g-h\|_2<\epsilon.$ As above, we can find a polynomial $p$ without a constant term such that $\|p-h\|_2<\epsilon.$ 
%In this case we obtain $\|f-p\|_2<3\epsilon$. 
%
%The function $x\mapsto \frac{x-t}{1-t}$ induces a topological isomorphism $\Phi:L^2[0,1]\to L^2[t,1]$ defined by 
%$$(\Phi f)(x)=f\left(\frac{x-t}{1-t}\right).$$
%Let us denote by $V_t$ the restriction of the Voterra operator to its invariant subspace $L^2[t,1].$
%It is not hard to see that $V_t$ satisfies 
%$$(V_t\Phi f)(s)=(1-t)\int_0^\frac{s-t}{1-t}f(v)dv=(1-t)(\Phi V f)(s)$$ for all $s\in [t,1].$ 
%Since $V$ has a dense range, $V_t$ also has a dense range. 
%\end {proof}

\begin {proposition}\label{zdravje}
Let $X$ be a Banach space and let $T$ be a compact unicellular operator on $X$. Then for an arbitrary subspace $\mathcal M$ from the triangularizing chain $\mathcal C$ of $T$ only one of the following statements hold. 
\begin {enumerate}
\item [(a)] $\mathcal M=\overline{T\mathcal M}.$
\item [(b)] $\mathcal M=\mathcal M_-$ and $\overline{T\mathcal M}$ is of codimension $1$ in $\mathcal M$. 
\item [(c)] $\overline{T\mathcal M}=\mathcal M_-$ is of codimension $1$ in $\mathcal M$. 
\end {enumerate}
Moreover, if $\mathcal C$ is a continuous chain, then $T$ is a strong quasi-affinity. 
\end {proposition}

It should be noted that Proposition \ref{zdravje} immediately implies that $V$ is a strong quasi-affinity and that $V$ is quasi-nilpotent by \cite[Example 7.2.5]{RadRos}.  

\begin {proof}
Assume that $\overline{T\mathcal M}$ is a proper subspace of $\mathcal M$. 
We claim that the dimension of the quotient space $\mathcal M/\overline{T\mathcal M}$ is  at most one. Otherwise, one can find two proper incomparable closed subspaces $\mathcal N_1$ and $\mathcal N_2$ of $\mathcal M$ that properly contain $\overline{T\mathcal M}$. 
Since $\mathcal N_1$ and $\mathcal N_2$ are also invariant under $T$, this contradicts unicellularity of $T$. 
Unicelullarity of $T$ also implies that $\overline{T\mathcal M}\in \mathcal C$. 
Since the codimension of $\mathcal M_-$ in $\mathcal M$ is also at most one, we have $\mathcal M_-=\overline{T\mathcal M}$ or $\mathcal M_-=\mathcal M.$

If $\mathcal C$ is a continuous chain, then it is obvious that (a) holds. 
\end {proof}

\begin {proposition}
Let $E$ be a Banach lattice and let $A$ be a positive compact unicellular operator on $E$ such that the unique triangularizing chain for $T$ consists of closed ideals of $E$. 
If a positive operator $B$ semi-commutes with $A$, then $A$ and $B$ are simultaneously triangularizable. 
\end {proposition}

\begin {proof}
Let $\mathcal C$ be a maximal chain of closed subspaces invariant under both $A$ and $B$. By the assumption, every subspace in $\mathcal C$ is an ideal of $E$.  Due to the maximality, $\mathcal C$ is a complete chain.
Suppose now that the dimension of $\mathcal J/\mathcal J_-$ is at least two for some $\mathcal J\in \mathcal C$.  Since the operator $A_{\mathcal J}+B_{\mathcal J}$ is ideal-irreducible on $\mathcal J/\mathcal J_-$ and semi-commutes with the operator $A_{\mathcal J}$, \cite[Corollary 3.5]{Gao} implies 
$$(A_{\mathcal J}+B_{\mathcal J})A_{\mathcal J}=A_{\mathcal J}(A_{\mathcal J}+B_{\mathcal J}),$$ so that we have 
$A_{\mathcal J}B_{\mathcal J}=B_{\mathcal J}A_{\mathcal J}.$
We claim that $A_{\mathcal J}$ is nonzero operator on $\mathcal J/\mathcal J_-$. Otherwise we have $A\mathcal J\subseteq \mathcal J_-$, and since Proposition \ref{zdravje} implies that the codimension of $\overline{A\mathcal J}$ is at most one, we have that the dimension of $\mathcal J/\mathcal J_-$ is at most one as well and we reached to a contradiction. 

The operator $A_{\mathcal J}$ is not a scalar multiple of the identity operator on $\mathcal J/ \mathcal J_-$ as otherwise $A$ would have at least two incomparable closed invariant subspaces which contradicts its unicellularity.  
By \cite[Corollary 2.4]{Sir} (if $E$ is a real lattice) or \cite{Lom} (if $E$ is a complex lattice), the operator $A_{\mathcal J}$ has a nontrivial hyperinvariant subspace. However, this is in contradiction with the maximality of the chain $\mathcal C$. Therefore, $\mathcal C$ is a maximal chain of closed subspaces of $E$. 
\end {proof}

Whenever one of the operators $A$ and $B$ in Proposition \ref{spek izrek} is the Volterra operator or a Donoghue operator, then the other operator does not need to have any compactness properties. 

\begin {corollary}\label{simpo}
Let $V$ be the Volterra operator on $L^2[0,1]$ and let $S$ be the Donoghue operator on $l^2.$
Then the following statements hold. 
\begin {enumerate}
\item [(a)] If a positive operator $T$ on $L^2[0,1]$ semi-commutes with $V$ then $V$ and $T$ are simultaneously ideal-triangularizable. 
\item [(b)] If a positive operator $T$ on $l^2$ semi-commutes with $S$, then $S$ and $T$ are simultaneously ideal-triangularizable. 
\end {enumerate}
\end {corollary}

In \cite[Example 3.3]{BDFRZ}, a Banach algebra generated by the Volterra operator $V$ and a multiplication operator $M$ on $L^2[0,1]$ defined by 
$(Mf)(x)=xf(x)$ was considered. Since the relation $MV-VM=V^2$ implies $MV\geq VM\geq 0$, Corollary \ref{mlekozavse} implies that $MV-VM$ is contained in the radical of the Banach algebra generated by $V$ and $M$. The authors \cite{BDFRZ} observed that this actually follows from the fact that $V$ and $M$ are simultaneously ideal-triangularizable with the unique triangularizing chain $\{L^2[t,1]\}_{t\in[0,1]}.$
The following result states that the Volterra operator is actually in the radical of the possibly much bigger Banach algebra generated by $\langle V]$ and $[V\rangle.$ 

\begin {theorem}
Let $V$ be the Volterra operator on $L^2[0,1]$ and let $S$ be the Donoghue operator on $l^2.$
Then $V$ is in the radical of the Banach algebra generated by $\langle V]$ and $[V\rangle$, and $S$ is in the radical of the Banach algebra 
generated by $\langle S]$ and $[S\rangle.$
\end {theorem}

\begin {proof}
By Corollary \ref{simpo} the set $\langle V]\cup[V\rangle$ is triangularizable, so that the closure of the algebra $\mathcal A$ generated $\langle V]$ and $[V\rangle$ is triangularizable as well. 
\cite[Theorem 7.2.4]{RadRos} implies that the operator $AV$ is quasi-nilpotent for all $A\in \mathcal A$ which finishes the proof. 

For the proof in the case of the Donoghue operator note first that it is quasi-nilpotent by \cite[Problem 96]{Hal} and its solution. \cite[Theorem 7.2.3]{RadRos} implies that every diagonal coefficient of the operator $S$ is zero.  Since $\langle S]\cup [S\rangle$ is triangularizable by Corollary \ref{simpo}, the Banach algebra $\mathcal B$ generated by $\langle S]\cup [S\rangle$ is triangularizable as well.   For all $B\in\mathcal B$ every diagonal coefficient of the operator $BS$ is zero, so that  the operator $BS$ is quasi-nilpotent by \cite[Theorem 7.2.3]{RadRos} and the proof is finished. 
\end {proof}

\begin {remark}
Let $V$ be the Volterra operator and let $\mathcal A$ be the Banach algebra generated by all compact operators in $\langle V]\cup [V\rangle$. 
Then $\mathcal A$ is triangularizable with the unique triangularizing chain $\{L^2[t,1]\}_{t\in [0,1]}.$ By  \cite[Corollary 7.2.4]{RadRos} every operator in $\mathcal A$ is quasi-nilpotent, so that $\mathcal A$ is a radical Banach algebra. 
\end {remark}

\begin {remark}
Suppose that $A$  and $B$ be positive operators on a Banach lattice $E$ with a positive commutator $AB-BA$. If one of the operators is power compact, then 
 $AB-BA$ is quasi-nilpotent by \cite[Corollary 3.2]{RD}. In particular, this implies that a positive commutator $AB-BA$ of positive power compact operators $A$ and $B$ is quasi-nilpotent which also follows from Theorem \ref{extensione}.
\end {remark}

We finish this paper with two open questions.

\begin {enumerate}
\item [(a)] How big is the Banach algebra generated by $\langle V]\cup [V\rangle$?

\item [(b)] Can we assume in Theorem \ref{extensione} that only one of the operators $A$ and $B$ is power compact?
\end {enumerate}

{\it Acknowledgments.} This work was supported in part by the Slovenian Research Agency. \\

\bigskip
		
\noindent
Marko Kandi\'{c}, Klemen \v Sivic: \\
     Faculty of Mathematics and Physics \\
     University of Ljubljana \\
     Jadranska 19 \\
     1000 Ljubljana \\
     Slovenia \\[1mm]
     E-mails : marko.kandic@fmf.uni-lj.si, klemen.sivic@fmf.uni-lj.si 

\end{document}